\documentclass{amsart}
\usepackage{amssymb, amsmath}
\usepackage[latin1]{inputenc}
\usepackage[final]{graphics}
\usepackage{color}
\newcommand{\bach}{\mathfrak{B}}

\DeclareMathOperator{\Id}{\rm Id}

\def\Chr#1#2#3{\Gamma_{#1#2}{}^#3{}}

\newcommand{\ce}{\mathcal{E}}

\newcommand{\GammaD}{\Gamma} 
\newcommand{\affcon}{\nabla} 

\begin{document}
\def\blue{\color{blue}}
\def\red{\color{red}}
\def\black{\color{black}}
\def\green{\color{green}}
\def\magenta{\color{magenta}}\
\newcommand{\grad}{\mbox{\rm grad}\,}
\newtheorem{theorem}{Theorem}[section]
\newtheorem{lemma}[theorem]{Lemma}
\newtheorem{remark}[theorem]{Remark}
\newtheorem{definition}[theorem]{Definition}
\newtheorem{corollary}[theorem]{Corollary}
\newtheorem{example}[theorem]{Example}
\newtheorem{ansatz}[theorem]{Ansatz}
\newtheorem{xcase}{Case}[subsection]
\newtheorem{ycase}{Case}[section]
\makeatletter
\renewcommand{\theequation}{%
\thesection.\alph{equation}} \@addtoreset{equation}{section}
\makeatother
\title[Bach Flat Manifolds]{Constructing Bach Flat Manifolds of signature $(2,2)$ using the modified Riemannian extension}
\author[Calvi\~no-Louzao]{E. Calvi\~no-Louzao, E. Garc\'{i}a-R\'{i}o, P. Gilkey, \\
I. Guti\'errez-Rodr\'iguez, and
R. V\'{a}zquez-Lorenzo}
\address{ECL: Conseller\'\i a de Cultura, Educaci\'on e Ordenaci\'on Universitaria, Edificio Administrativo San Caetano, 15781 Santiago de Compostela, Spain}
\email{{estebcl@edu.xunta.es}}
\address{EGR-IGR: Faculty of Mathematics, University of Santiago de Compostela, 15782 Santiago de Compostela, Spain}
\email{eduardo.garcia.rio@usc.es; ixcheldzohara.gutierrez@usc.es}
\address{PG: Mathematics Department, \; University of Oregon, \;\; Eugene \; OR 97403, \; USA}
\email{gilkey@uoregon.edu}
\address{RVL: Department of Mathematics, IES de Ribadeo Dionisio Gamallo, 27700 Ribadeo, Spain}
\email{{ravazlor@edu.xunta.es}}
\thanks{Supported by projects ED431F 2017/03, and MTM2016-75897-P (Spain).}

\begin{abstract} We  use the modified Riemannian extension of an affine surface to construct
Bach flat manifolds. As all these examples
are VSI (vanishing scalar invariants), we shall construct scalar invariants which are not of Weyl
type to distinguish them. We illustrate this phenomena in the context of homogeneous affine surfaces.
\end{abstract}
\maketitle

\section{Introduction}
The gravitational field equations in General Relativity arise from the Hilbert-Einstein functional variation of the metric. 
Different modifications to General Relativity have been extensively studied in a quest for a quantum 
theory of gravity. Conformal gravity is a theory of gravity in four dimensions which is invariant under
conformal transformations (hence sensitive to angles but not distances).  Indeed, any Weyl transformation
of the metric, $g_{ij} \mapsto \Psi^2(x)g_{ij}$, is an exact symmetry of this action.
In the simplest form, its action consists of the $L^2$-norm of Weyl curvature tensor 
$S_{\operatorname{conf}}:=\int d^4 x \,\sqrt{g}\, \|W\|^2$ (see, for example  \cite{Maldacena, Mannheim} and 
the references therein for more information). 
The field equations of four-dimensional conformal gravity therefore require the vanishing of the Bach tensor. 

It is a well known fact that the solutions of Einstein gravity are also solutions of conformal
gravity. But conformal gravity has other solutions. Indeed, since the Bach tensor vanishes identically for Einstein metrics and since the Bach tensor is conformally invariant, 
the conformally Einstein metrics provide
a large class of solutions.
A more intriguing problem is the construction of strict solutions to conformal gravity, meaning those
which are neither conformally Einstein nor half conformally flat. We refer to the work in
\cite{LN10, LLPV13,NP01} for examples of strictly Bach flat four-manifolds (see also
\cite{AGS13, BDE16}). Nevertheless, there is a paucity of Bach flat manifolds
which are not conformally Einstein.

A modification of the classical Patterson-Walker Riemannian extension  \cite{PW52} was used in
\cite{CGGV} to provide a new source of strictly Bach flat metrics which support gradient Ricci solitons.
This construction requires the existence of a background affine surface admitting a parallel nilpotent tensor 
field, which is a rather restrictive condition (see \cite{CGGGV18}).
In this paper, we shall generalize the construction of \cite{CGGV} to characterize Bach flat 
Riemannian extensions of affine surfaces admitting a nilpotent structure. We use the
 Cauchy-Kovalevski Theorem to show that any such modified Riemannian extension can be locally deformed to a Bach flat one in the real analytic setting.  It is worth 
emphasizing
  that any real analytic affine surface gives rise to a
(locally defined) Bach flat Riemannian extension.
We show that all these metrics have vanishing scalar curvature invariants (VSI). 
For that reason, we shall
introduce suitable invariants which are not of Weyl type to distinguish different classes; these
invariants are, of course, of interest in their own right.

Let  $\mathcal{N}=(N,g)$ be a pseudo-Riemannian manifold and let
${}^g\nabla$ be  the associated Levi Civita connection. With our sign convention, the
curvature operator takes the form
$R(X,Y):={}^g\nabla_X{}^g\nabla_Y-{}^g\nabla_Y{}^g\nabla_X-{}^g\nabla_{[X,Y]}$. Let $W$ be the 
Weyl conformal curvature tensor, and let $\rho$ be the Ricci tensor.
Adopt the {\it Einstein convention} and sum over repeated indices.
The {\it Bach tensor} of $\mathcal{N}$ is the conformally invariant,  trace-free, and divergence-free symmetric $(0,2)$-tensor field  given by setting:
$$
\mathfrak{B}_{ij}:
={}^g\nabla^k\, {}^g\nabla^\ell W_{kij\ell}+\textstyle\frac12\rho^{k\ell}W_{kij\ell}\,.
$$
We say that $\mathcal{N}$ is {\it Bach flat} if $\mathfrak{B}=0$.
We shall be interested in the case that $\mathcal{N}$ has neutral signature $(2,2)$. 
Let $\mathcal{M}=(M,\nabla)$ be an affine surface. If $(x^1,x^2)$ are local coordinates on $M$, let $(y_1,y_2)$ be the associated
dual coordinates on the cotangent bundle where a 1-form is expressed as $\omega=y_1dx^1+y_2dx^2$. 
Let $T=T^r{}_i\partial_{x^r}\otimes dx^i$
be a tensor field
of type~$(1,1)$ on $M$ (i.e. an endomorphism of the tangent bundle $TM$)
and let $\Phi_{ij}$ be a symmetric 2-tensor on $M$. The associated modified Riemannian extension
\begin{equation}\label{E1.a}
g_{T,\nabla,\Phi}= 2\, dx^i\circ dy_i\!+\!\left\{  \frac{1}{2} y_ry_s (T^r{}_i T^s{}_j + T^r{}_j T^s{}_i)
 \! -\!2y_k\Gamma_{ij}{}^k\!+\! \Phi_{ij}\right\} dx^i\circ dx^j
\end{equation}
is invariantly defined and independent of the particular system of local coordinates (see, for example, the discussion
in \cite{CGGV09}). Let
$$
\mathcal{S}_T:=\{P\in M:T(P)=\lambda(P)\Id\}\text{ and }\mathcal{O}_T:=M-S_T\,.
$$
The space $\mathcal{S}_T$ is the set of points where $T$ is a scalar multiple of the identity;
$\mathcal{O}_T$ is the complementary space.
We will establish the following result in Section~\ref{S2}.
\begin{theorem} \label{T1.1} 
Let $\mathcal{M}=(M,\nabla)$ be an affine surface, let $T$ be a tensor of type $(1,1)$, and let $\Phi$ be a symmetric $2$-tensor. Let
$\mathcal{N}:=(T^*M,g_{T,\nabla,\Phi})$.
\begin{enumerate}
\item If $M=\mathcal{S}_T$, then $\mathcal{N}$ is half conformally flat 
and hence Bach flat.
\item $\mathcal{O}_T$ is an open subset of $M$. 
If $P\in\mathcal{O}_T$ and if $\mathfrak{B}(P)=0$, then $T(P)^2=0$.
\item If $T$ is nilpotent on $M$ and if $T(P)\ne0$, then there exist local coordinates near $P$ so that
$T=\partial_{x^1}\otimes dx^2$. The
following assertions are equivalent in such a coordinate system.
\begin{enumerate}
\item $\mathcal{N}$ is Bach flat.
\item $\Gamma_{11}{}^2=0$ and
$(\Gamma_{11}{}^1)^2-\Gamma_{11}{}^1\Gamma_{12}{}^2+\partial_{x^1}(\Gamma_{11}{}^1-\Gamma_{12}{}^2)=0$.
\end{enumerate}\end{enumerate}
\end{theorem}

\begin{remark}\label{R1.2}
\rm 
We note that the auxiliary tensor $\Phi$ plays no role in the analysis. 
If $\phi(x^1,x^2)$ is a smooth function, set
$\phi^{(1,0)}=\partial_{x^1}\phi$, $\phi^{(2,0)}=\partial_{x^1x^1}\phi$, and so forth. 
We can express the conditions
of Assertion~(3b) in the form
$$
\Gamma_{11}{}^2=0,\quad \Gamma_{11}{}^1=-\phi^{(1,0)},\quad
\Gamma_{12}{}^2=\Gamma_{11}{}^1+c\cdot e^\phi\
$$
for smooth functions $c=c(x^2)$ and $\phi=\phi(x^1,x^2)$. Assertion~(1) generalizes a result of \cite{CGGV}
which considered Bach flat manifolds in the context of parallel tensor fields~$T$.
\end{remark}

If $T$ is a scalar multiple of the identity, then $\mathcal{N}$ is half conformally flat. We focus, therefore,
on the case $T$ is nilpotent henceforth and assume, unless otherwise noted, that $M=\mathcal{O}_T$. 
We work locally. Fix $P\in M$
and a local system of coordinates defined near $P$. We wish to find $0\ne T$ nilpotent so that $\mathcal{N}$ is Bach flat. Since
either $T^1{}_2(P)\ne0$ or $T^2{}_1(P)\ne0$, we assume for the sake of definiteness that $T^1{}_2(P)\ne0$. This implies that we may expand
$T$ near $P$ in the form
\begin{equation}\label{E1.b}
T=\alpha(x^1,x^2)\left(\begin{array}{cc}\xi(x^1,x^2)&1\\-\xi^2(x^1,x^2)&-\xi(x^1,x^2)\end{array}\right).
\end{equation}

\begin{definition}\label{D1.3}
We introduce the following operators:
\medskip

\par\noindent$\quad
\mathcal{P}_1(\xi):=-\xi ^{(1,0)}+\xi\,\xi ^{(0,1)} +{\Chr221} \xi ^3-(2 {\Chr121} -{\Chr222}) \xi ^2
+({\Chr111}  -2 {\Chr122}) \xi
+{\Chr112}\,, $
\medbreak\noindent$\quad
\mathcal{P}_2(\xi,\alpha):=\alpha \alpha^{(2,0)}+ \xi^2 \alpha \alpha^{(0,2)}-2 \xi \alpha  \alpha^{(1,1)}
+(\alpha^{(1,0)})^2
+\xi^2 (\alpha^{(0,1)})^2
-2 \xi  \alpha^{(1,0)} \alpha^{(0,1)}
$
\smallbreak\noindent$\qquad
-\alpha \alpha^{(1,0)} \left(2 \xi^{(0,1)}-5 \Chr221 \xi^2+2(4 \Chr121- \Chr222) \xi-3 \Chr111+2 \Chr122\right)
$
\smallbreak\noindent$\qquad
+\alpha \alpha^{(0,1)} 
\left(
2 \xi \xi^{(0,1)}
-6 \Chr221 \xi^3+(10 \Chr121-3 \Chr222) \xi^2- 4(\Chr111- \Chr122)\xi-\Chr112
\right)
$
\smallbreak\noindent$\qquad
+6   \xi^4 \alpha^2 (\Chr221)^2
-2 \xi^3 \alpha^2\left(
(\Chr221)^{(0,1)}+9 \Chr121 \Chr221
-3 \Chr221 \Chr222 
\right)  
$

\smallskip

\par\noindent$\qquad
-  \xi^2 \alpha^2
\left(
4 \Chr221 \xi^{(0,1)}
-3 (\Chr121)^{(0,1)}
-2 (\Chr221)^{(1,0)}+(\Chr222)^{(0,1)}
\right.
$
\smallbreak\noindent$\qquad\left.
\phantom{-  \xi^2 \alpha^2(}
-12 (\Chr121)^2
-(\Chr222)^2
-7 \Chr111 \Chr221
+7  \Chr121 \Chr222
+9 \Chr122 \Chr221 
\right)
$
\smallbreak\noindent$\qquad
+\xi  \alpha^2
\left(
2(3 \Chr121-\Chr222) \xi^{(0,1)}-(\Chr111)^{(0,1)}
-3 (\Chr121)^{(1,0)}
+(\Chr122)^{(0,1)}
\right.
$
\smallbreak\noindent$\qquad
\phantom{+\xi  \alpha^2(}
\left.
+(\Chr222)^{(1,0)}-2 (\Chr111-\Chr122)(4 \Chr121-\Chr222) +4 \Chr112 \Chr221
\right)
$
\smallbreak\noindent$\qquad  
- \alpha^2\left(2 (\Chr111-\Chr122) \xi^{(0,1)}-(\Chr111)^{(1,0)}+(\Chr122)^{(1,0)}\right.
$
\smallbreak\noindent$\qquad
\phantom{- \alpha^2(}
\left.-(\Chr111)^2+\Chr111 \Chr122
+3 \Chr112 \Chr121-\Chr112 \Chr222
\right) \,.
$
\end{definition}

Theorem~\ref{T1.1} permits us to construct connections so the Riemannian extension is Bach flat once the
nilpotent endomorphism is given. We will establish the following result in Section~\ref{S3} which focuses on 
the reverse problem of constructing nilpotent endomorphisms so the Riemannian extension is Bach flat once
the connection is given; this is, in a certain sense, a more natural question.

\begin{theorem}\label{T1.4}
Let $(M,\nabla)$ be an affine surface. Let $T$ have the form given in Equation~(\ref{E1.b}) and let $\Phi$ be arbitrary.
The modified Riemannian extension $(T^*M,g_{T,\nabla,\Phi})$ of Equation~(\ref{E1.a}) is
Bach flat if and only if $\alpha$ and $\xi$ are solutions to the
partial differential equations $\mathcal{P}_1(\xi)=0$ and $\mathcal{P}_2(\xi,\alpha)=0$.
\end{theorem}

Suppose $\mathcal{M}$ is real analytic. 
The operator $\mathcal{P}_1(\xi)$ of Definition~\ref{D1.3} takes the form:
$$
\mathcal{P}_1(\xi)=-\xi^{(1,0)}+\xi\xi^{(0,1)}+f(\xi,\GammaD)\,.
$$
Given a real analytic function $\xi_0(x^2)$, the Cauchy-Kovalevski Theorem shows that there is
a unique local solution to the equation $\mathcal{P}_1(\xi)=0$ with $\xi(0,x^2)=\xi_0(x^2)$. 
Once $\xi$ is determined, the operator  $\mathcal{P}_2(\xi,\alpha)$ of Definition~\ref{D1.3} takes the form
\begin{eqnarray*}
\mathcal{P}_2(\xi,\alpha)&=&\alpha\alpha^{(2,0)} { -2\xi\alpha}\alpha^{(1,1)}
+{ \xi^2\alpha}\alpha^{(0,2)}+F(\alpha,d\alpha;\GammaD,{ d\GammaD};\xi,{ d\xi})\,.
\end{eqnarray*}
Given real analytic functions $\alpha_0(x^2)$ and $\alpha_1(x^2)$,
there exists a unique local solution to the equation
$\mathcal{P}_2(\xi,\alpha)=0$ with $\alpha(0,x^2)=\alpha_0(x^2)$
and $\alpha^{(1,0)}(0,x^2)=\alpha_1(x^2)$.
Thus given $\affcon$, there are many nilpotent $T$ so that $\mathcal{N}$ is Bach flat in this setting; the auxiliary
tensor $\Phi$ plays no role in the analysis. 
In Section~\ref{S4}, we exhibit some specific examples of Bach flat manifolds.

Let $0\neq T=T^j{}_i(x^1,x^2)$ be  a nilpotent tensor field of type $(1,1)$ as in Equation~(\ref{E1.b}). 
A straightforward calculation shows that
\[
W^-(E_1^-,E_1^-)=\frac{1}{2} \alpha(x^1,x^2)^2 (\xi(x^1, x^2)^2+1)^2,
\quad
W^+(E_1^+,E_2^+)=-2 \rho_a^\affcon(\partial_{x^1},\partial_{x^2})\,.
\]
Therefore, $W^-$ is always non-null  and the non-symmetry of $\rho^\affcon$ guarantees that $(T^*M,g_{T,\affcon,\Phi})$ is not half conformally flat.

In Section~\ref{S5}, we explore the geometry of the Riemannian extension defined by a nilpotent tensor. 
We say a pseudo-Riemannian manifold $\mathcal{N}=(N,g)$ is VSI (vanishing scalar invariants) if all the scalar Weyl invariants
(i.e. invariants formed by a complete contraction of indices in the Riemann curvature tensor $R_{ijk\ell}$
and its covariant derivatives) vanish 
(see \cite{CHMMB14, H} and references therein for more information and examples of VSI manifolds).
Let $\tau$ be the scalar curvature.

\begin{theorem}\label{T1.5}
Let $\mathcal{N}=(T^*M,g_{T,\affcon,\phi})$ with $T\neq 0$. The following assertions are equivalent:
\par\quad{\rm(1)} $\mathcal{N}$ is VSI.\quad
{\rm(2)} $\|R\|^2=\|\rho\|^2=0$.\quad{\rm(3)} $\|\rho\|^2=\tau=0$.
\quad{\rm(4)} $T$ is nilpotent.
\end{theorem}

In Example~\ref{E5.1}, we will show that the conditions $\|R\|^2=\tau^2=0$ do not suffice
to show that $T$ is nilpotent nor does the condition $\|\rho\|^2=0$ suffice to show that $T$ is nilpotent.
In Section~\ref{S6}, we construct invariants of $\mathcal{N}$
which are not of Weyl type. Both invariants rely upon the fact that $\mathcal{N}$ is a Walker
manifold, i.e. that there exists a parallel totally isotropic 2-plane $\mathfrak{V}$. Examples
are presented; the auxiliary endomorphism $\Phi$ enters for the first time in the analysis.

\section{The proof of Theorem~\ref{T1.1}}\label{S2}

A direct computation shows that if $T=f\Id$ for $f\in C^\infty(M)$,
then $\mathcal{N}$ is half conformally flat \cite{CGGV09}, and thus $\mathfrak{B}=0$;
this establishes Assertion~(1) of Theorem~\ref{T1.1}.
In Sections \ref{S2.1} and \ref{S2.2}, we establish the second and third assertions, respectively,
of Theorem~\ref{T1.1}. Let $\mathfrak{B}=0$. By Theorem~\ref{T1.1}~(3),
we may decompose $M=\mathcal{S}_T\dot\cup\mathcal{O}_T$
as the disjoint union of the set of points where $T$ is a scalar multiple of the identity 
and the set of points where $T$ is nilpotent and has non-trivial
Jordan normal form. In the real analytic setting, if $\mathcal{O}_T$ is non-empty and if $M$
is connected, then $\mathcal{O}_T$ is dense in $M$ and
$T$ is always nilpotent.  In Section~\ref{S2.3}, we provide an example in the smooth category 
where this observation fails.

\subsection{The proof of Theorem~\ref{T1.1}~(2)}\label{S2.1}
Let $\Theta_{ijk\ell}$ be the coefficient of $y_iy_j$
in $\mathfrak{B}_{k\ell}$. A straightforward computation shows that $\Theta_{ijk\ell}$ 
is a polynomial which is homogeneous of degree 6 in the $T^u{}_v$ variables for
$k,\ell\in\{ 1,2\}$ and zero otherwise; 
the Christoffel symbols and their derivatives,
the auxiliary endomorphism $\Phi$ and its derivatives, and the derivatives of $T$ do not appear in these terms. 
Consequently, $\Theta=\{\Theta_{ijk\ell}\}$ is tensorial.
Assume that $\mathcal{N}$ is Bach flat. This implies $\Theta(T)=0$. We suppose $P\in\mathcal{O}_T$, i.e.
$T(P)$ is not a scalar multiple of the identity. Let $\{\lambda_1,\lambda_2\}$ be the
(possibly complex) eigenvalues of $T(P)$. We
can make a complex linear change of coordinates in the $\{x^1,x^2\}$ variables to put $T$ in
upper triangular form;
this induces a corresponding dual complex linear change of coordinates in the $\{y_1,y_2\}$ variables. This is,
of course, just Jordan normal form. Thus we
may assume that:
\begin{equation}\label{E2.aa}
T(P):=\left(\begin{array}{cc}\lambda_1&\varepsilon\\0&\lambda_2\end{array}\right)\,.
\end{equation}
\subsection*{Suppose $\lambda_1\ne\lambda_2$}
We compute:
\begin{eqnarray*}
&&\Theta_{1111}(T(P))=\textstyle\frac16\lambda_1^2(\lambda_1-\lambda_2)^2(\lambda_1^2+\lambda_1\lambda_2-5\lambda_2^2),\\
&&\Theta_{2222}(T(P))=\textstyle\frac16\lambda_2^2(\lambda_1-\lambda_2)^2(-5\lambda_1^2+\lambda_1\lambda_2+\lambda_2^2)\,.
\end{eqnarray*}
Note that the parameter $\varepsilon$ does not appear; these two terms are not sensitive to the
precise Jordan normal form but only to the eigenvalues. Since $\mathfrak{B}=0$ and since 
$\lambda_1-\lambda_2\ne0$, we obtain 
\begin{eqnarray}
&&\lambda_1^2(\lambda_1^2+\lambda_1\lambda_2-5\lambda_2^2)=0,\label{E2.a}\\
&&\lambda_2^2(-5\lambda_1^2+\lambda_1\lambda_2+\lambda_2^2)=0.\label{E2.b}
\end{eqnarray}
If $\lambda_1=0$, then $\lambda_2\ne0$ and Equation~(\ref{E2.b}) fails;
if $\lambda_2=0$, then $\lambda_1\ne0$ and Equation~(\ref{E2.a}) fails.
Thus $\lambda_1\ne0$ and $\lambda_2\ne0$ and we obtain
$$
\lambda_1^2+\lambda_1\lambda_2-5\lambda_2^2=0\text{ and }
-5\lambda_1^2+\lambda_1\lambda_2+\lambda_2^2=0\,.
$$
Subtracting these two identities yields $6\lambda_1^2-6\lambda_2^2=0$. 
Since $\lambda_1\ne\lambda_2$, we have
$\lambda_1=-\lambda_2$ so $-5\lambda_1^2=0$
and again $\lambda_1=\lambda_2=0$ which is false.

\subsection*{Suppose $\lambda_1=\lambda_2$}
We obtain $\Theta_{1122}(T)=-3\varepsilon^2\lambda_1^4$; this term is sensitive to the Jordan normal form.
Since $T(P)$ is not a scalar multiple of the identity, $\varepsilon\ne0$.
Thus $\lambda_1=0$ and $T(P)$ is nilpotent. If we perturb an endomorphism which is not a scalar multiple of the
identity, we obtain a similar endomorphism. This shows that $\mathcal{O}_T$ is open and completes the proof of
the second assertion of Theorem~\ref{T1.1}.\qed 

\subsection{The proof of Theorem~\ref{T1.1}~(3)}\label{S2.2}
Let $T$ be a nilpotent tensor of Type~(1,1) on a surface $M$. Assume $T(P)\ne0$.
Let $Z_2$ be a non-zero vector field which is defined near $P$
so that $TZ_2\ne0$. Then $Z_1:=TZ_2$ spans $\ker(T)$ and $\{Z_1,Z_2\}$ is a local frame for $TM$. 
If $\{Z^1,Z^2\}$ is the dual frame for $T^*M$, then
$T=Z_1\otimes Z^2$. Choose local coordinates $\{z^1,z^2\}$ so
that $Z_1=\partial_{z^1}$. Then $T\partial_{z^2}=f\partial_{z^1}$ for some non-zero function $f$. Let $X_1=f\partial_{z^1}$ 
and $X_2=g\partial_{z^1}+\partial_{z^2}$ where $g$ remains to be determined.
Then $TX_2=X_1$. We have
$[X_1,X_2]=(f\partial_{z^1}g-g\partial_{z^1}f-\partial_{z^2}f)\partial_{z^1}$. Solve the ODE
$\partial_{z^1}g=f^{-1}\{g\partial_{z^1}f+\partial_{z^2}f\}$ with initial condition $g(0,z^2)=0$.
This ensures $[X_1,X_2]=0$. Since $\{X_1,X_2\}$ are linearly independent, we can choose local coordinates
so $\partial_{x^1}=X_1$ and $\partial_{x^2}=X_2$. We then have $T=\partial_{x^1}\otimes dx^2$.

Suppose $\mathfrak{B}=0$.
Examining $\mathfrak{B}_{11}$ yields $\Gamma_{11}{}^2=0$. Examining $\mathfrak{B}_{22}$ yields the remaining
relation of Assertion~(3b). A direct computation shows that if the relations of Assertion~(3b) are satisfied,
then the Riemannian extension is Bach flat.\hfill\qed

\subsection{The relation between $\mathcal{S}_T$ and $\mathcal{O}_T$}\label{S2.3}
Let $M=\mathbb{R}^2$, 
let $\alpha(x^2)$ be a smooth real valued function which vanishes to infinite order
at $x^2=0$ and which is positive for $x^2\ne0$.\
Impose the conditions of Theorem~\ref{T1.1}~(3b) and assume that
$\Gamma_{11}{}^2=0$ and
$(\Gamma_{11}{}^1)^2-\Gamma_{11}{}^1\Gamma_{12}{}^2+\partial_{x^1}(\Gamma_{11}{}^1-\Gamma_{12}{}^2)=0$.
 Let
$$
T(x^1,x^2)=\left\{\begin{array}{lll}
\left(\begin{array}{cc}\alpha(x^2)&0\\0&\alpha(x^2)\end{array}\right)&\text{if}&x^2\le0\\[.2in]
\left(\begin{array}{cc}\ \ \ 0\ \ \ &\alpha(x^2)\\0&0\end{array}\right)&\text{if}&x^2\ge0.
\end{array}\right\}\,.
$$
One may then compute that $\mathfrak{B}=0$ so this yields a Bach flat manifold where the
Jordan normal form of $T$ changes at $x^2=0$. Furthermore, if we only assume that $\alpha$ is $C^k$
for $k\ge2$, we still obtain a solution; thus there is no hypo-ellipticity present when considering the
solutions to the equations $\mathfrak{B}=0$.

\section{The proof of Theorem~\ref{T1.4}}\label{S3}
We suppose $T$ is a nilpotent tensor field of type~$(1,1)$. 
Then $\operatorname{Trace}(T)=0$ and $\det(T)=0$. If
we  assume that $T^1{}_2(P)\ne0$, then $T$ has the form given in Equation~(\ref{E1.b}). 
A direct computation shows $\mathfrak{B}(\partial_{x^k},\partial_{y_j})=0$ and
 $\mathfrak{B}(\partial_{y_i},\partial_{y_j})=0$, and thus only $\mathfrak{B}_{11}$, $\mathfrak{B}_{12}$, and $\mathfrak{B}_{22}$, where $\mathfrak{B}_{ij}=\mathfrak{B}(\partial_{x^i},\partial_{x^j})$, are relevant. We observe that
$$\begin{array}{ll}
\text{Coefficient}[\mathfrak{B}_{11},\alpha^{(2,0)}]=-4\alpha\xi^2,&
\text{Coefficient}[\mathfrak{B}_{12},\alpha^{(2,0)}]=-4\alpha\xi,\\[0.05in]
\text{Coefficient}[\mathfrak{B}_{22},\alpha^{(2,0)}]=-4\alpha\,.
\end{array}$$
We therefore define $\mathfrak{Q}_1:=\mathfrak{B}_{11}-\mathfrak{B}_{12}\xi$, $\mathfrak{Q}_2:=\mathfrak{B}_{11}-\mathfrak{B}_{22}\xi^2$, and $\mathfrak{Q}_3:=2\mathfrak{Q}_1-\mathfrak{Q}_2$. 
We may then express $\mathfrak{Q}_3=-4\alpha^2(\mathcal{P}_1)^2$ and thus the vanishing of $\mathfrak{Q}_3$ is equivalent to the vanishing of
$\mathcal{P}_1$. We set $\mathcal{P}_1=0$ and express $\xi^{(1,0)}=F_{(1,0)}(\xi,\GammaD,\xi^{(0,1)})$. 
Differentiating this relation permits us to express
$\xi^{(1,1)}=F_{(1,1)}(\xi,\GammaD,d\GammaD,\xi^{(0,1)},\xi^{(0,2)})$ and 
$\xi^{(2,0)}=F_{(2,0)}(\xi,\GammaD,d\GammaD,\xi^{(0,1)},\xi^{(0,2)})$. Substituting these
relations then yields $\mathfrak{Q}_1=0$ and $\mathfrak{Q}_2=0$. Thus only $\mathfrak{B}_{11}$ plays a role. Substituting these relations permits us to express
$\mathfrak{B}_{11}=-4\xi^2\mathcal{P}_2$, $\mathfrak{B}_{12}=-4\xi\mathcal{P}_2$, and $\mathfrak{B}_{22}=-4\mathcal{P}_2$. The desired result now follows.
\qed

\section{Examples of Bach flat manifolds}\label{S4}
A pseudo-Riemannian manifold $(N^n,g)$ is said to be \emph{(locally) conformally Einstein} if every point $P\in N$
has an open neighborhood $\mathcal{U}$ and a positive smooth function $\varphi$ defined on 
$\mathcal{U}$ such that $(\mathcal{U},\bar{g} :=\varphi^{-2} g)$ is Einstein.
Brinkmann \cite{Brinkmann24} showed that a manifold is conformally Einstein if and only if
the equation
\begin{equation}\label{eq:Conformally Einstein}
(n-2)\operatorname{Hes}_\varphi +\varphi\,\rho- \frac{1}{n}\{(n-2)\Delta\varphi+\varphi\,\tau\}g=0
\end{equation}
has a positive solution. Although the conformally Einstein equation is quite simple,
integrating it is surprisingly difficult (see \cite{Kuhnel-Rademacher} and references therein for more information).
It was shown in \cite{GN, Kozameh-Newman-Tod85} that any four-dimensional conformally Einstein manifold satisfies  
\begin{equation}\label{eq:18-b}
\operatorname{div}_4 W(\cdot,\cdot,\cdot)-W(\cdot,\cdot,\cdot,\nabla \log\varphi)=0
\quad\text{and}\quad\bach=0\,.
\end{equation} 
We say that $(N,g)$ is \emph{weakly-generic} if the Weyl tensor is injective viewed as a map
from $TN$ to $\bigotimes^3 TN$. In this setting, the relations of Equation~(\ref{eq:18-b}) suffice to
imply $(N,g)$ is conformally Einstein. 
 
The existence of a null distribution $\mathfrak{V}$ on a four-dimensional 
pseudo-Riemannian manifold of neutral signature 
defines a natural orientation. This orientation is characterized by the fact that
if $\{u, v\}$ is any basis for $\mathfrak{V}$, then
the bivector $u\wedge v$ is self-dual (see \cite{rm08}).

Let $\pi:T^*M\rightarrow M$ be the natural projection. Then $\mathfrak{V}:=\ker\pi_*$
is a null distribution. We give $T^*M$ the orientation induced by $\mathfrak{V}$; 
self-duality and anti-self-duality are no longer interchangeable in this context.
For the remainder of this section, let $\mathcal{N}:=(T^*M,g_{T,\affcon,\Phi})$ be the
Riemannian extension..
We define a local orthonormal frame
of signature $(++--)$  for the tangent bundle by setting:
\medbreak\qquad
$e_1:=\partial_{x^{1}}\!+\!\frac{1}{2} (1-(g_{T,\affcon,\Phi})_{11}) \partial_{y_1}$,
\smallbreak\qquad
$e_2:=\partial_{x^{2}}\!- \!(g_{T,\affcon,\Phi})_{12} \partial_{y_1} \!+\! \frac{1}{2} (1-(g_{T,\affcon,\Phi})_{22}) \partial_{y_2}$,
\smallbreak\qquad
$e_3:=\partial_{x^{1}}\!-\!\frac{1}{2} (1+(g_{T,\affcon,\Phi})_{11}) \partial_{y_1}$,
\smallbreak\qquad
$e_4:=\partial_{x^{2}}\!-\!(g_{T,\affcon,\Phi})_{12} \partial_{y_1}\!-\!\frac{1}{2} (1+(g_{T,\affcon,\Phi})_{22}) \partial_{y_2}$.
\medbreak\noindent
The volume form
$e^1\wedge e^2\wedge e^3\wedge e^4$ gives the orientation determined by $\mathfrak{V}$.
Let $\Lambda^2_\pm$ be the spaces of self-dual and anti-self-dual 2-forms. Let
$$\begin{array}{ll}
e^{ij}:=e^i\wedge e^j,&
E_1^\pm := \frac{1}{\sqrt{2}}(e^{12} \pm e^{34}),\\[0.05in]
E_2^\pm := \frac{1}{\sqrt{2}}(e^{13} \pm e^{24}),&
E_3^\pm: = \frac{1}{\sqrt{2}}(e^{14} \mp e^{23}).
\end{array}$$
Then $\{ E_1^\pm, E_2^\pm,
E_3^\pm\}$ is an orthonormal basis for $\Lambda^2_\pm$.
Let $W^\pm_{ij}:=W^\pm(E^\pm_i,E^\pm_j)$ 
be the components of the self-dual and anti-self-dual parts of the Weyl curvature tensor.
Let $0\neq T=T^j{}_i(x^1,x^2)$ be  a nilpotent tensor field of type $(1,1)$ which
has the form given in Equation~(\ref{E1.b}). 
A straightforward calculation shows that
\begin{eqnarray*}
&&W^-(E_1^-,E_1^-)=\textstyle\frac{1}{2} \alpha(x^1,x^2)^2 (\xi(x^1, x^2)^2+1)^2,\\
&&W^+(E_1^+,E_2^+)=-2 \rho_a^\affcon(\partial_{x^1},\partial_{x^2})\,.
\end{eqnarray*}
Consequently $\mathcal{N}$ is never self-dual.
Furthermore, if the Ricci tensor $\rho^\nabla$ of $(M,\nabla)$  is non-symmetric, then
$\mathcal{N}$ is not anti-self-dual.
Since $\mathcal{N}$ is not weakly-generic,
we must work directly with Equation \eqref{eq:Conformally Einstein}. Let
\begin{eqnarray*}
&&\textstyle
\ce:=2\operatorname{Hes}_\varphi +\varphi\,\rho- \frac{1}{4}\{2\Delta\varphi+\varphi\,\tau\}g,\\
&&\widetilde\ce:=\operatorname{div}_4 W(\cdot,\cdot,\cdot)-W(\cdot,\cdot,\cdot,\nabla \log\varphi)\,.
\end{eqnarray*}
Let $\phi\in C^\infty(T^*M)$. One computes that
\[
\ce(\partial_{y_1},\partial_{y_1})=2 \partial_{y_1y_1}\varphi,\quad
\ce(\partial_{y_1},\partial_{y_2})=2 \partial_{y_1y_2}\varphi,\quad
\ce(\partial_{y_2},\partial_{y_2})=2 \partial_{y_2y_2}\varphi\,.
\]
Consequently any solution of Equation~\eqref{eq:Conformally Einstein} has the form 
$\varphi=\iota X+\psi\circ\pi$, where $\iota X$ is the evaluation of a vector field 
$X=A\partial_{x^1}+B\partial_{x^2}$ on $M$ and $\psi\in\mathcal{C}^\infty(M)$.

\subsection{The locally homogeneous setting}
An affine surface $\mathcal{M}=(M,\affcon)$ is
said to be {\it locally homogeneous} if given any two points of $M$, there is a local diffeomorphism
from a neighborhood of the first point to a neighborhood of the second commuting
with $\affcon$. Opozda \cite{Op04} has classified the local geometry of such structures
dividing them into three classes; the classes are not exclusive and we refer to
\cite{BGG18} for further details.

\begin{theorem}[Opozda]\label{Th4.1}
 Let $\mathcal{M}=(M,\affcon)$ be a locally homogeneous affine surface which is not flat. Then at least one of the following
 three possibilities holds which describe the local geometry:
 \begin{itemize}
  \item[($\mathcal{A}$)] There exists a coordinate atlas so the Christoffel symbols
  $\GammaD_{ij}{}^k$ are constant.
  \item[($\mathcal{B}$)] There exists a coordinate atlas so the Christoffel symbols have the form
  $\GammaD_{ij}{}^k=(x^1)^{-1}C_{ij}{}^k$ for $C_{ij}{}^k$ constant and $x^1>0$.
  \item[($\mathcal{C}$)] $\affcon$ is the Levi-Civita connection of a metric of constant Gauss
  curvature.
\end{itemize}\end{theorem}

We now provide some examples of Bach flat manifolds where the underlying affine structure
is homogeneous.

\begin{example}\label{E4.2}\rm
Let the Christoffel symbols
 $$\begin{array}{cccccc}
 \GammaD_{11}{}^1=0,&\GammaD_{11}{}^2=0,&\GammaD_{12}{}^1=1,&\GammaD_{12}{}^2=1,&
 \GammaD_{22}{}^1=0,&\GammaD_{22}{}^2=0\,,
 \end{array}$$
define a  type $\mathcal{A}$ affine surface. The Ricci tensor of $\affcon$ is 
$-(dx^1-dx^2)^2$ so this structure is not flat. We can exhibit nilpotent tensor fields
of Type~(1,1) which give rise to Bach flat structures as follows. If
$\alpha_i\in C^\infty(\mathbb{R})$, let
 \begin{eqnarray*}
&&T:=\alpha_2(x^2)\sqrt{e^{2x^1}+\alpha_1(x^2)}\left\{\partial_{x^1}\otimes d{x^2}\right\},\\
&&
  \widetilde T:=\alpha_2(x^1)\sqrt{e^{2x^2}+\alpha_1(x^1)}\left\{\partial_{x^2}\otimes d{x^1}\right\}\,.
 \end{eqnarray*}
The endomorphisms $T$ and $\tilde T$ lead to Bach flat manifolds.
\end{example}

\medbreak
We now use Theorem~\ref{T1.1} to construct Bach flat manifolds. Note that the Ricci tensor
of any Type~$\mathcal{A}$ structure is symmetric.

\begin{example}\rm Let $\affcon$ be a Type~$\mathcal{A}$ structure on $\mathbb{R}^2$. 
Let $0\ne T\in M_2(\mathbb{R})$ be nilpotent.
 Make a linear change of coordinates to ensure $T=\partial_{x^1}\otimes dx^2$. 
 Since the Christoffel symbols are constant, the condition of Theorem~\ref{T1.1}~(3b)
 becomes $\GammaD_{11}{}^2=0$ and
 $(\GammaD_{11}{}^1)^2-\GammaD_{11}{}^1\GammaD_{12}{}^2=0$.
 If $\Phi=0$, then $\mathcal{N}$ is anti-self-dual.
  \begin{enumerate}
 \item If $\GammaD_{11}{}^1=\GammaD_{11}{}^2=0$ and $\GammaD_{12}{}^2= 0$, 
 let  $\varphi= y_1\operatorname{e}^{-\GammaD_{12}{}^1 x^2}$.
  \item If  $\GammaD_{11}{}^1=\GammaD_{11}{}^2=0$ and $\GammaD_{12}{}^2\neq 0$,
let $\varphi=\operatorname{e}^{-\GammaD_{12}{}^2 x^1 + \GammaD_{12}{}^1 x^2}$.
\item If $\GammaD_{11}{}^2=0$ and $\GammaD_{11}{}^1=\GammaD_{12}{}^2$, let
$\varphi=y_1\operatorname{e}^{-\GammaD_{12}{}^1 x^2} $. 
\end{enumerate}
 One has that  $\varphi^{-2}g_{T,\affcon,\Phi}$ is Einstein, and thus $\mathcal{N}$
 is conformally Einstein.
Next suppose $\Phi\ne0$, $\GammaD_{11}{}^1=\GammaD_{11}{}^2=0$, and
$\GammaD_{12}{}^2\neq 0$. One has
$$\partial_{y_1} W^+(E_1^+,E_1^+)=-\partial_{x^1}\Phi_{11}(x^1,x^2)+2\GammaD_{12}{}^2\Phi_{11}(x^1,x^2)\,.
$$
 A straightforward calculation shows that the possible conformal factors have
the form $\varphi=\kappa\operatorname{e}^{-\GammaD_{12}{}^2x^1+\GammaD_{12}{}^1x^2}$
where $\kappa\in\mathbb{R}$ . In this situation,
$$
\ce(\partial_{x^2},\partial_{x^2})=\varphi(x^1,x^2,y_1,y_2) \Phi_{11}(x^1,x^2)\,.
$$
 Hence if
$\Phi_{11}(x^1,x^2)\neq 0$ and
$\partial_{x^1}\Phi_{11}(x^1,x^2)-2\GammaD_{12}{}^2\Phi_{11}(x^1,x^2)\neq 0$ (or equivalently,
if $\Phi_{11}(x^1,x^2)\neq e^{2\GammaD_{12}{}^2 x^1} P(x^2)$), then 
 $\mathcal{N}$ is strictly Bach flat. Moreover, since $(\affcon_{\partial_{x^1}}T)\partial_{x^2}=-\GammaD_{12}{}^2\partial_{x^1}$, we have  $\affcon T\neq 0$ in this case.
\end{example}

\begin{example}\label{E4.4}\rm Let $\affcon$ be a Type~$\mathcal{B}$ structure on 
$\mathbb{R}^+\times\mathbb{R}$. This means that the Christoffel symbols of $\affcon$ 
take the form $\GammaD_{ij}{}^k=(x^1)^{-1}C_{ij}{}^k$. Let 
 $0\ne T\in M_2(\mathbb{R})$ be nilpotent. The map $(x^1,x^2)\rightarrow(x^1,ax^2+bx^1)$ defines
 an action of the ``$ax+b$" group on such structures and modulo such an action, we may assume $T$ takes
 one of the following two forms:
 \begin{enumerate}
  \item $T=\partial_{x^1}\otimes dx^2$. A direct computation shows that $\mathcal{N}$ is Bach flat if and only if
  $C_{11}{}^2=0$ and $(C_{11}{}^1-1)(C_{11}{}^1-C_{12}{}^2)=0$.
  \item $T=\partial_{x^2}\otimes dx^1$. A direct computation shows that $\mathcal{N}$ is Bach flat if and only if 
  $C_{22}{}^1=0$ and $C_{22}{}^2(C_{12}{}^1-C_{22}{}^2)=0$.
\end{enumerate}
\smallbreak\noindent {\bf Case (1)} Let $C_{11}{}^2=0$ and 
$(C_{11}{}^1-1)(C_{11}{}^1-C_{12}{}^2)=0$. The Ricci tensor
$\rho^\affcon$ is symmetric if and only if $C_{12}{}^1 + C_{22}{}^2=0$. We distinguish cases.
\begin{enumerate}
\item Suppose $C_{11}{}^2=0$ and $C_{11}{}^1=1$.
We note that  $\affcon T\neq 0$ in this case since 
$(\affcon_{\partial_{x^2}}T)\partial_{x^1}=-\frac{C_{12}{}^2}{x^1}\partial_{x^1}$ and $
(\affcon_{\partial_{x^1}}T)\partial_{x^2}=\frac{1-C_{12}{}^2}{x^1}\partial_{x^1}$.
Further assume that  $\rho^\affcon$ is not symmetric (i.e.  $C_{12}{}^1 + C_{22}{}^2\neq 0$).
Then $\mathcal{N}$  is not half conformally flat.  A  straightforward calculation shows that any 
solution of the conformally Einstein equation \eqref{eq:Conformally Einstein} takes the form 
$\varphi=(x^1)^{2-C_{12}{}^2} P(x^2)$. In this setting 
$2(x^1)^3\,   \widetilde\ce(\partial_{x^1},\partial_{x^2},\partial_{x^1})
 =  C_{12}{}^1 (5 - 4 C_{12}{}^2 )- C_{22}{}^2$.
Hence,  $C_{12}{}^1 + C_{22}{}^2\neq 0$ and $ C_{22}{}^2\neq C_{12}{}^1(5 -  4 C_{12}{}^2) $ imply that $\mathcal{N}$ is strictly Bach flat. 
If $\rho^\affcon$ is  not symmetric and $ C_{22}{}^2= C_{12}{}^1(5 -  4 C_{12}{}^2)$, we distinguish two cases depending on whether $C_{12}{}^2$ equals $1$ or not.  
\begin{enumerate}
\item If $C_{12}{}^2=1$ then a straightforward calculation shows that $\mathcal{N}$ is conformally Einstein if and only if
$$
\Phi_{11}(x^1,x^2)=A(x^2)-\frac{B(x^2)}{x^1}+\frac{4C_{22}{}^1}{(x^1)^2}\,.
$$
The possible conformal factors take  the form $\varphi=x^1 P(x^2)$ where
$
2 P''(x^2) + A(x^2)P(x^2) =0$ and
$2 C_{12}{}^1 P'(x^2) + B(x^2) P(x^2) =0$.
\item If $C_{12}{}^2\neq 1$, then  $\mathcal{N}$ is conformally Einstein if and only if the
deformation tensor $\Phi$ satisfies
$\Phi_{11}(x^1,x^2)=4(x^1)^{-2}(C_{22}{}^1+2(C_{12}{}^1)^2(C_{12}{}^2-1))$ and
the conformal factor satisfies $\varphi=\kappa (x^1)^{2-C_{12}{}^2}$
where $\kappa\in\mathbb{R}$.
\end{enumerate}
\item Suppose $C_{11}{}^2=0$ and $C_{12}{}^2=C_{11}{}^1$.
In  this  case
$\affcon T$ is determined by 
\begin{eqnarray*}
&&(\affcon_{\partial_{x^2}}T)\partial_{x^1}=-\frac{C_{11}{}^1}{x^1}\partial_{x^1},
\\
&&(\affcon_{\partial_{x^2}}T)\partial_{x^2}=\frac{C_{12}{}^1-C_{22}{}^2}{x^1}\partial_{x^1}+\frac{C_{11}{}^1}{x^1}\partial_{x^2}\,.
\end{eqnarray*}
If  $C_{12}{}^1 + C_{22}{}^2\neq 0$,  then $\mathcal{N}$ is not half conformally flat and, moreover,  
any solution of \eqref{eq:Conformally Einstein} takes the form $\varphi=(x^1)^{C_{11}{}^1} P(x^2)$.
In such a case, 
\[
\ce(\partial_{x^1},\partial_{x^2}) =  (x^1)^{-2}  (C_{22}{}^2-C_{12}{}^1) \varphi(x^1,x^2,y_1,y_2)\,.
\]
Suppose $\rho^\affcon$  is not symmetric. If $C_{22}{}^2\neq C_{12}{}^1$,
then $\mathcal{N}$  is strictly Bach flat. On the other hand,
 if $C_{22}{}^2 = C_{12}{}^1$ then a straightforward calculation shows that 
 $\mathcal{N}$ is conformally Einstein if and only if 
 $$
 \Phi_{11}(x^1,x^2)=A(x^2)-\frac{B(x^2)}{x^1}+\frac{2C_{22}{}^1(C_{11}{}^1+1)}{(x^1)^2}
 $$
 and  any possible conformal factor takes  the form $\varphi=(x^1)^{C_{11}{}^1} P(x^2)$ where
\[\qquad\qquad
2 P''(x^2) + A(x^2)P(x^2) =0\text{ and }
2 C_{12}{}^1 P'(x^2) + B(x^2) P(x^2) =0\,.
\]
\end{enumerate}

\smallbreak\noindent{\bf Case (2)} Suppose $C_{22}{}^1=0$ and
$C_{22}{}^2(C_{12}{}^1-C_{22}{}^2)=0$. The Ricci tensor $\rho^\affcon$ is symmetric 
if and only if $C_{12}{}^1=0$. Again, we distinguish cases.
\begin{enumerate}
\item Suppose $C_{22}{}^1=0$, $C_{22}{}^2=0$, and
$\rho^\affcon$ is not symmetric (i.e. $C_{12}{}^1\neq 0$). Then $\mathcal{N}$ is not half
conformally flat. Moreover, a straightforward calculation shows that the solutions of
Equation~\eqref{eq:Conformally Einstein}  take the form
$\varphi=\operatorname{e}^{-\Gamma_{12}{}^1 x^2} P(x^1)$. In such a case, 
\[
\partial_{x^2}\left(
(x^1)^3 \operatorname{e}^{\Gamma_{12}{}^1 x^2}\ce(\partial_{x^1},\partial_{x^2})
\right) = -4 (C_{12}{}^1)^2 P(x^1)
\]
so $\mathcal{N}$  is strictly Bach flat. Moreover, $\affcon T\neq 0$ since $(\affcon_{\partial_{x^1}}T)\partial_{x^2}=-\frac{C_{12}{}^1}{x^1}\partial_{x^2}$.
\item
Suppose $C_{22}{}^1=0$ , $C_{22}{}^2= C_{12}{}^1$, and the Ricci tensor $\rho^\affcon$ is  
non-symmetric. Then $\mathcal{N}$ is not half conformally flat. The possible conformal factors 
take the form $\varphi=\operatorname{e}^{\Gamma_{12}{}^1 x^2} P(x^1)$. Then $(x^1)^2\ce(\partial_{x^1},\partial_{x^2}) = -2 C_{12}{}^1\varphi(x^1,x^2,y_1,y_2)$
so $\mathcal{N}$  is strictly Bach flat.
Moreover, $(\affcon_{\partial_{x^1}}T)\partial_{x^2}=-\frac{C_{12}{}^1}{x^1}\partial_{x^2}$
so $\affcon T\neq 0$. 
\end{enumerate}
\end{example}

\begin{example}\label{E4.5}
\rm Impose the relations of Remark~\ref{R1.2} and set
$$
\Gamma_{11}{}^2=0,\quad
\Gamma_{11}{}^1=-\partial_{x^1}\beta,\quad
\Gamma_{12}{}^2=-\partial_{x^1}\beta+ce^\beta\text{ for }  c=c(x^2).
$$
We consider the nilpotent endomorphism $T^1{}_1=0$, $T^2{}_2=0$, $T^2{}_1=0$, 
and $T^1{}_2=e^f$. This yields Bach flat manifold if and only if
$$
0=f^{(1,0)}  \left(2ce^{\beta  }+\beta ^{(1,0)} \right)-2 (f^{(1,0)})^2-f^{(2,0)}\,.
$$
In particular, any function $f=f(x^2)$ will work in this instance.
\end{example}

\begin{example}\label{E4.6}\rm
We now impose further relations interchanging the roles of the indices to specialize the remaining 3 Christoffel symbols:
$$\begin{array}{llll}
\Gamma_{11}{}^2=0,&\Gamma_{11}{}^1=-\partial_{x^1}\beta,&\Gamma_{12}{}^2=-\partial_{x^1}\beta+ce^\beta,&\text{ for } c=c(x^2),\\
\Gamma_{22}{}^1=0,&\Gamma_{22}{}^2=-\partial_{x^2}\gamma,&\Gamma_{12}{}^1=-\partial_{x^2}\gamma+de^\gamma,&
\text{ for } d=d(x^1).
\end{array}$$
Then in addition to the solution of Example~\ref {E4.5} we have 
$e^{\tilde f}\partial_{x^2}\otimes dx^1$where
$$
0=\tilde f^{(0,1)} (2 d e^{\gamma  }+\gamma ^{(0,1)} )-2 (\tilde f^{(0,1)})^2-\tilde f^{(0,2)}\,.
$$
\end{example}

\section{The proof of Theorem~\ref{T1.5}}\label{S5}

 Let $\mathcal{N}=(T^*M,g_{T,\nabla,\Phi})$.
Clearly if $\mathcal{N}$ is VSI, then $\|R\|^2=\|\rho\|^2=\tau=0$. 
Thus Assertion~(1) of Theorem~\ref{T1.5} implies Assertions~(2) and  (3).
In Section~\ref{S5.1}, we will show Assertion~(2) or Assertion~(3)
 imply Assertion~(4), i.e. $\|R\|^2=\|\rho\|^2=0$ or $\|\rho\|^2=\tau=0$ implies
 $T$ is nilpotent. In Example~\ref{E5.1}, we exhibit a structure where
 $\|R\|^2=\|\tau\|^2=0$ and $T$ is not nilpotent. We will also exhibit
 a structure where $\|\rho\|^2=0$ and $T$ is not nilpotent.
Although the fact that $T$ is nilpotent implies $\mathcal{N}$ is VSI follows from the results in
\cite{CHMMB14,H}, we include a direct proof in Section~\ref{S5.3} for sake of completeness.

\subsection{Vanishing scalar invariants}\label{S5.1}

A direct computation shows that $\tau$ is a quadratic polynomial in the components of $T$ and
that $\|R\|^2$ and $\|\rho\|^2$ are fourth order polynomials in the components
of $T$; the other variables do not enter. 
Let $\{\lambda_1,\lambda_2\}$ be the eigenvalues of $T$.
We make a complex linear change of coordinates in the $(x^1,x^2)$
variables to put $T(P)$ in upper triangular form of Equation~(\ref{E2.aa}).
 The parameter $\varepsilon$ plays no role and we obtain at $P$ that
\begin{equation}\label{E5.x}\begin{array}{l}
\tau=2 \left(\lambda_1^2+\lambda_1\lambda_2+\lambda_2^2\right),
\qquad
\|R\|^2=4(\lambda_1^4+\lambda_1^2\lambda_2^2+\lambda_2^4),\\
\|\rho\|^2=2\lambda_1^4+2\lambda_1^3\lambda_2+\lambda_1^2\lambda_2^2+2\lambda_1\lambda_2^3+2\lambda_2^4.
\end{array}\end{equation}
\smallbreak\noindent{\bf Assertion~(2) implies Assertion~(4).} Assume $\|R\|^2=0$ and $\|\rho\|^2=0$.
If the eigenvalues are real, then the vanishing of $\|R\|^2$ implies
$\lambda_1=0$ and $\lambda_2=0$ so $T$ is
nilpotent. Thus we assume the eigenvalues are complex so $\lambda_2=\bar\lambda_1\ne0$. Set $\lambda_1=re^{i\theta}$ and 
$\lambda_2=re^{-i\theta}$ for $r\ne0$. The equations in question are homogeneous so we may assume without
loss of generality $r=1$. We have
$$0=\|\rho\|^2-\frac12\|R\|^2=2\lambda_1^3\lambda_2-\lambda_1^2\lambda_2^2+2\lambda_1\lambda_2^3\,.$$
Dividing this equation by $\lambda_1\lambda_2$ yields
$0=2\lambda_1^2-\lambda_1\lambda_2+2\lambda_2^2$.
Setting $\lambda_1=e^{i\theta}$ and $\lambda_2=e^{-i\theta}$ we obtain
$\textstyle0=e^{4i\theta}+1+e^{-4i\theta}$ so $\cos(4\theta)=-\frac12$ and
$\textstyle0=2e^{2i\theta}-1+2e^{2i\theta}$ so $\cos(2\theta)=\frac14$.
This case can not occur since the angle addition formulas would yield
$\textstyle-\frac12=\cos(4\theta)=2\cos^2(2\theta)-1=\frac1{8}-1$.
\smallbreak\noindent{\bf Assertion~(3) implies Assertion~(4).} Assume $\tau=0$ and $\|\rho\|^2=0$.
We compute  $\|\rho\|^2-\frac12\tau^2=-\lambda_1\lambda_2(2\lambda_1+\lambda_2)
(\lambda_1+2\lambda_2)$. Setting $\tau=0$ then yields $\lambda_1=\lambda_2=0$.\qed

\begin{example}\label{E5.1}\rm Let $r(x^1,x^2)>0$ be an arbitrary smooth function
and let $\theta$ be constant. Set
$$
T=r(x^1,x^2)\left(\begin{array}{cc}\cos(\theta)&\sin(\theta)\\-\sin(\theta)&\cos(\theta)\end{array}\right)\,.
$$
We obtain $\lambda_1=re^{i\theta}$ and $\lambda_2=re^{-i\theta}$. This example is
not nilpotent and we have
$$\begin{array}{ll}
\tau=2r^2(2\cos(2\theta)+1),&
\|R\|^2=4r^4(2\cos(4\theta)+1),\\[0.05in]
\|\rho\|^2=r^4(4\cos(4\theta)+4\cos(2\theta)+1)\,.
\end{array}$$
\begin{enumerate}
\item If $\theta=\frac\pi3$, then $\cos(4\theta)=\cos(2\theta)=-\frac12$ so
$\|R\|^2=\tau=0$. 
\item If $\theta=0$, then $4\cos(4\theta)+4\cos(2\theta)+1=9$. Similarly, if $\theta=\frac\pi4$,
then $4\cos(4\theta)+4\cos(2\theta)+1=-3$. Thus by the Intermediate Value Theorem we may
choose $0<\theta<\frac\pi4$ so that $4\cos(4\theta)+4\cos(2\theta)+1=0$ and
consequently $\|\rho\|^2=0$. In fact, one can determine $\theta$ exactly;
one can take
$\theta=\frac12\arctan(\frac{\sqrt 7+1}{\sqrt7-1})$.
\end{enumerate}\end{example}

\begin{remark}
\rm If $(a,b,c)\in\mathbb{R}^3$, let $\kappa_{a,b,c}:=a\tau^2+b\|R\|^2+c\|\rho\|^2$ define
a single quadratic invariant. There exists a non-empty 
open subset $\mathcal{O}$ of $\mathbb{R}^3$
so that if $(a,b,c)\in\mathcal{O}$, then $\kappa_{a,b,c}(\mathcal{N})=0$ implies $T$ is
nilpotent; thus this single quadratic curvature invariant characterizes VSI manifolds in the
setting at hand. For example, the quadratic scalar invariants 
$4\|\rho\|^2-3\tau^2$ or $\| R\|^2 -\frac{88}{5}\|\rho\|^2 +\frac{56}{5}\tau^2$ vanish if and only if $T$ is nilpotent, and thus $\mathcal{N}$ is  VSI.
\end{remark}

\subsection{Nilpotent $T$}\label{S5.3}
Set $x_3=y_1$ and $y_2=x_4$ to have a consistent notation in what follows.
Assume that $T$ is nilpotent.
By Theorem~\ref{T1.1}~(3), we may choose coordinates so $T=\partial_{x^1}\otimes dx^2$. 
Let $g=g_{T,\nabla,\Phi}$. Then $\{g^{ij},{}^g\Gamma_{ij}{}^k,R_{abcd;e_1\dots e_k}\}$ are
polynomial expressions in the fiber coordinates whose coefficients depend on
$\{\Gamma_{ij}{}^k,\Phi_{ij}\}$ and their derivatives with respect to $x^1$ and $x^2$. 
In such a coordinate system, one computes that the possibly
non-zero components of the tensor $g^{ij}$, of the Christoffel symbols, and of the curvature $R$ are, up to the usual 
$\mathbb{Z}_2$ symmetries given by
\begin{equation} \label{E5.a}
\begin{array}{ccccccccc}
g^{13},&g^{24},&g^{33},&g^{34},&g^{44},&{}^g\Gamma_{11}{}^1,&{}^g\Gamma_{11}{}^2,&{}^g\Gamma_{11}{}^3,\\[0.05in]
{}^g\Gamma_{11}{}^4,&{}^g\Gamma_{12}{}^1,&{}^g\Gamma_{12}{}^2,&
{}^g\Gamma_{12}{}^3,&{}^g\Gamma_{12}{}^4,&{}^g\Gamma_{13}{}^3,&{}^g\Gamma_{13}{}^4,&{}^g\Gamma_{14}{}^3,\\[0.05in]
{}^g\Gamma_{14}{}^4,&{}^g\Gamma_{22}{}^1,&{}^g\Gamma_{22}{}^2,&{}^g\Gamma_{22}{}^3,&{}^g\Gamma_{22}{}^4,
&{}^g\Gamma_{23}{}^3,&{}^g\Gamma_{23}{}^4,&{}^g\Gamma_{24}{}^3,\\[0.05in]
{}^g\Gamma_{24}{}^4,&
R_{1212},&R_{1213},&R_{1214},&
R_{1223},&R_{1224},&R_{2323}.
\end{array}\end{equation}

Of particular interest is the fact that $R_{ 2323}=-1$. Let $\mathfrak{o}(\cdot)$ be the maximal
order of an expression in the dual variables $\{y_1=x_3,y_2=x_4\}$. Thus if $\mathfrak{o}(\cdot)=0$, these
variables do not occur, if $\mathfrak{o}(\cdot)=1$, the expression is linear in the variables $\{x_3,x_4\}$, and so forth.
In other words, we define $\mathfrak{o}(x_3)=\mathfrak{o}(x_4)=1$ and extend $\mathfrak{o}$ to a derivation. If
$\mathfrak{o}(R_{ijk\ell})=2$, then $R_{ijk\ell}$ is at most quadratic in $\{x_3,x_4\}$; if $\mathfrak{o}(R_{ijk\ell})=1$, then $R_{ijk\ell}$ is at most linear in $\{x_3,x_4\}$;
and if $\mathfrak{o}(R_{ijk\ell})=0$, then $R_{ijk\ell}$ does not involve $\{x_3,x_4\}$. We have
$$
\begin{array}{llll}
\mathfrak{o}({}^g\Gamma_{11}{}^1)=0,&\mathfrak{o}({}^g\Gamma_{11}{}^2)=0,&\mathfrak{o}({}^g\Gamma_{11}{}^3)=1,&\mathfrak{o}({}^g\Gamma_{11}{}^4)=2,\\
\mathfrak{o}({}^g\Gamma_{12}{}^1)=0,&\mathfrak{o}({}^g\Gamma_{12}{}^2)=0,&\mathfrak{o}({}^g\Gamma_{12}{}^3)=1,&\mathfrak{o}({}^g\Gamma_{12}{}^4)=2,\\
\mathfrak{o}({}^g\Gamma_{13}{}^3)=0,&\mathfrak{o}({}^g\Gamma_{13}{}^4)=0,&\mathfrak{o}({}^g\Gamma_{14}{}^3)=0,&\mathfrak{o}({}^g\Gamma_{14}{}^4)=0,\\
\mathfrak{o}({}^g\Gamma_{22}{}^1)=1,&\mathfrak{o}({}^g\Gamma_{22}{}^2)=0,&\mathfrak{o}({}^g\Gamma_{22}{}^3)=2,&\mathfrak{o}({}^g\Gamma_{22}{}^4)=2,\\
\mathfrak{o}({}^g\Gamma_{23}{}^3)=0,&\mathfrak{o}({}^g\Gamma_{23}{}^4)=1,&\mathfrak{o}({}^g\Gamma_{24}{}^3)=0,&\mathfrak{o}({}^g\Gamma_{24}{}^4)=0,\\
\mathfrak{o}(R_{1212})=2,&\mathfrak{o}(R_{1213})=1,&\mathfrak{o}(R_{1214})=0,&\mathfrak{o}(R_{1223})=1,\\
\mathfrak{o}(R_{1224})=1,&\mathfrak{o}(R_{2323})=0.
\end{array}$$

We define the {\it defect} by setting
\begin{eqnarray*}
&&\mathfrak{d}( {}^g\Gamma_{ij}{}^k)=-\sum_{n=1}^2\{\delta_{i,n}+\delta_{j,n}-\delta_{k,n}\}+\sum_{n=3}^4\{\delta_{i,n}+\delta_{j,n}-\delta_{k,n}\},\\
&&\mathfrak{d}(R_{i_1i_2i_3i_4;i_5\dots i_\nu}):=\sum_{n=1}^\nu\{\delta_{i_n,3}+\delta_{i_n,4}-\delta_{i_n,1}-\delta_{i_n,2}\}\,.
\end{eqnarray*}
In brief, we count, with multiplicity, each lower index  `$1$' or `$2$' with a $-1$ and `$3$' or `$4$' with a $+1$ and reverse the sign for upper indices.
This will play an important role in contracting indices subsequently.
We then set $\mathfrak{x}=\mathfrak{o}+\mathfrak{d}$ and compute:
\begin{equation}\label{E5.b}
\begin{array}{llll}
\mathfrak{x}({}^g\Gamma_{11}{}^1)=-1,&\mathfrak{x}({}^g\Gamma_{11}{}^2)=-1,&\mathfrak{x}({}^g\Gamma_{11}{}^3)=-2,&\mathfrak{x}({}^g\Gamma_{11}{}^4)=-1,\\
\mathfrak{x}({}^g\Gamma_{12}{}^1)=-1,&\mathfrak{x}({}^g\Gamma_{12}{}^2)=-1,&\mathfrak{x}({}^g\Gamma_{12}{}^3)=-2,&\mathfrak{x}({}^g\Gamma_{12}{}^4)=-1,\\
\mathfrak{x}({}^g\Gamma_{13}{}^3)=-1,&\mathfrak{x}({}^g\Gamma_{13}{}^4)=-1,&\mathfrak{x}({}^g\Gamma_{14}{}^3)=-1,&\mathfrak{x}({}^g\Gamma_{14}{}^4)=-1,\\
\mathfrak{x}({}^g\Gamma_{22}{}^1)=0,&\mathfrak{x}({}^g\Gamma_{22}{}^2)=-1,&\mathfrak{x}({}^g\Gamma_{22}{}^3)=-1,&\mathfrak{x}({}^g\Gamma_{22}{}^4)=-1,\\
\mathfrak{x}({}^g\Gamma_{23}{}^3)=-1,&\mathfrak{x}({}^g\Gamma_{23}{}^4)=0,&\mathfrak{x}({}^g\Gamma_{24}{}^3)=-1,&\mathfrak{x}({}^g\Gamma_{24}{}^4)=-1,\\
\mathfrak{x}(R_{1212})=-2,&\mathfrak{x}(R_{1213})=-1,&\mathfrak{x}(R_{1214})=-2,&
\mathfrak{x}(R_{1223})=-1,\\\mathfrak{x}(R_{1224})=-1,&\mathfrak{x}(R_{2323})=0.
\end{array}\end{equation}

\begin{lemma}\label{L5.1}
Suppose that $R_{i_1i_2i_3i_4;i_5\dots i_\nu}\ne0$. Then $\mathfrak{x}(R_{i_1i_2i_3i_4;i_5\dots i_\nu})\le0$.
Furthermore,
$\mathfrak{x}(R_{i_1i_2i_3i_4;i_5\dots i_\nu})=0$ if and only if  $R_{i_1i_2i_3i_4;i_5\dots i_\nu}=\pm R_{2323}$.
\end{lemma}
\begin{proof} Let $R_{ijk\ell}\ne0$. By Equation~(\ref{E5.b}), $\mathfrak{x}(R_{ijk\ell})\le0$ with equality
if and only if $R_{ijk\ell}=\pm R_{2323}$.
This establishes the result if $\nu=4$. Next we suppose $\nu=5$ and examine ${}^g\nabla R$. 
Suppose $R_{i_1i_2i_3i_4;n}$ is non-zero as a polynomial formula and
that $\mathfrak{x}(R_{i_1i_2i_3i_4;n})\ge0$. We argue for a contradiction. We expand
\begin{eqnarray*}
R_{i_1i_2i_3i_4;n}&=&\partial_{n}R_{i_1i_2i_3i_4}-\sum_a{}^g\Gamma_{ni_1}{}^aR_{ai_2i_3i_4}
-\sum_a {}^g\Gamma_{ni_2}{}^aR_{i_1ai_3i_4}\\
&&-\sum_a{}^g\Gamma_{ni_3}{}^aR_{i_1i_2ai_4}-\sum_a{}^g\Gamma_{ni_4}{}^aR_{i_1i_2i_3a}\,.
\end{eqnarray*}
There are several possibilities that can ensure
$R_{i_1i_2i_3i_4;n}$ is potentially non-zero as a polynomial formula which we examine seriatim.
\subsection*{Case 1} We could have that $\partial_nR_{i_1i_2i_3i_4}\ne0$. If $n\in\{1,2\}$, then
\begin{eqnarray*}
&&\mathfrak{d}(R_{i_1i_2i_3i_4;n})=\mathfrak{d}(R_{i_1i_2i_3i_4})-1<0,\quad
\mathfrak{o}(\partial_nR_{i_1i_2i_3i_4})\le\mathfrak{o}(R_{i_1i_2i_3i_4}),\\
&&\mathfrak{x}(\partial_nR_{i_1i_2i_3i_4})\le\mathfrak{x}(R_{i_1i_2i_3i_4})-1<0\,.
\end{eqnarray*}
If $n\in\{3,4\}$, then necessarily $\mathfrak{o}(R_{i_1i_2i_3i_4})>0$ to ensure $R_{i_1i_2i_3i_4}$ in fact depends on $(x_3,x_4)$. Thus $R_{i_1i_2i_3i_4}\ne R_{2323}$. We have
\begin{eqnarray*}
&&\mathfrak{d}(R_{i_1i_2i_3i_4;n})=\mathfrak{d}(R_{i_1i_2i_3i_4})+1,\quad
\mathfrak{o}(\partial_nR_{i_1i_2i_3i_4})\le\mathfrak{o}(R_{i_1i_2i_3i_4})-1,\\
&&\mathfrak{x}(R_{i_1i_2i_3i_4;n})\le\mathfrak{x}(R_{i_1i_2i_3i_4})<0\,.
\end{eqnarray*}
Thus in any event, $\mathfrak{x}(R_{i_1i_2i_3i_4;n})<0$ which contradicts our initial assumption.

\subsection*{Case 2} Suppose ${{}^g\Gamma_{ni_1}{}^a}R_{ai_2i_3i_4}\ge0$ for some $a$; 
the remaining 4 cases involving ${ {}^g\Gamma_{ni_2}{}^a}R_{i_1ai_3i_4}$, ${ {}^g\Gamma_{ni_3}{}^a}R_{i_1i_2ai_4}$,
and ${ {}^g\Gamma_{ni_4}{}^a}R_{i_1i_2i_3a}$ are similar. As
$0\ge\mathfrak{x}({ {}^g\Gamma_{ni_1}{}^a})$ and $0\ge\mathfrak{x}(R_{ai_2i_3i_4})$,
$$
0\ge\mathfrak{x}({ {}^g\Gamma_{ni_1}{}^a})+\mathfrak{x}(R_{ai_2i_3i_4})=\mathfrak{x}({ {}^g\Gamma_{ni_1}{}^a}R_{ai_2i_3i_4})\ge0\,.
$$
Thus  $\mathfrak{x}({ {}^g\Gamma_{ni_1}{}^a})=0$ and $\mathfrak{x}(R_{ai_2i_3i_4})=0$. By Equation~(\ref{E5.b}),
$R_{ai_2i_3i_4}=\pm R_{2323}$ so $a\in\{2,3\}$. Since $\mathfrak{x}({ {}^g\Gamma_{ni_1}{}^a})=0$,  Equation~(\ref{E5.b}) then shows
${ {}^g\Gamma_{ni_1}{}^a}\in\{{ {}^g\Gamma_{22}{}^1},{ {}^g\Gamma_{23}{}^4}\}$ so $a\in\{1,4\}$. This is not possible.

We conclude that if $R_{i_1i_2i_3i_4;i_5}\ne0$, then $\mathfrak{x}(R_{i_1i_2i_3i_4;i_5})<0$. The
argument for ${ {}^g\nabla}^\ell R$ for $\ell\ge2$ now proceeds by induction; we do not have the additional complexity
involved in considering variables $R_{i_1i_2i_3i_4;i_5\dots i_{\nu}}$ where $\mathfrak{x}(R_{i_1i_2i_3i_4;i_5\dots i_{\nu}})=0$.
\end{proof}

Let $\mathcal{W}$ be a Weyl scalar invariant formed from the curvature tensor and its covariant derivatives. 
By equation~(\ref{E5.a}), we can contract an index `1' against an index `3' and an index `2' against an index `4'. We can also contract indices
$\{3,4\}$ against $\{3,4\}$. Thus
if $A=R_{i_1i_2i_3i_4;i_5\dots i_{\nu}}\dots$ is a monomial, then
$$
\operatorname{deg}_1(A)\le\operatorname{deg}_3(A)\text{ and }\operatorname{deg}_2(A)\le\operatorname{deg}_4(A)\,.
$$
The inequality can, of course, be strict as we can also contract an index $3$ or $4$ against an index $3$ or $4$.  
Thus $\mathfrak{d}(A)\ge0$. Since $\mathfrak{o}(A)\ge0$, this implies $\mathfrak{x}(A)\ge0$. 
By Lemma~\ref{L5.1},
$\mathfrak{x}(A)\le0$. Consequently, $\mathfrak{x}(A)=0$ so $A$ is a power of $R_{2323}$. 
As we can not
contract an index `2' against an index `3',  $\mathcal{W}=0$.\hfill\qed

\section{Invariants which are not of Weyl Type}\label{S6}
Let $\mathcal{M}=(M,\nabla)$ be an affine surface, let $T$ be nilpotent, and let
$$
\mathcal{N}:=(T^*M,g_{T,\nabla,\Phi})
$$
be the associated VSI Riemannian extension.
We do not impose the condition that $\mathcal{N}$ is Bach flat.
We begin by decomposing the curvature of $R^g$ and the associated Ricci tensor.
Choose coordinates so $T=\partial_{x^1}\otimes dx^2$ and keep the notation of Section~\ref{S5} so that $y_1=x_3$ and $y_2=x_4$. 
Let $\{R,\rho\}$ be the curvature operator and Ricci tensor of $\mathcal{N}$
and let $\{R^\nabla,\rho^\nabla,\rho^\nabla_a,\rho^\nabla_s\}$ be the curvature operator, Ricci tensor, alternating Ricci tensor, and symmetric Ricci tensor  of $\mathcal{M}$. 
Let $\mathfrak{V}:=\operatorname{Span}\{\partial_{x_3},\partial_{x_4}\}$ be the ``vertical"
and let $\mathfrak{H}:=\operatorname{Span}\{\partial_{x^1},\partial_{x^2}\}$ be the ``horizontal" space.
These are, of course, not invariantly defined. We may then decompose
$$
R(X,Y)=\left(\begin{array}{cc}
R_{\mathfrak{H}}^{\mathfrak{H}}=\left(\begin{array}{cc}R_{XY1}{}^1&R_{XY2}{}^1\\R_{XY1}{}^2&R_{XY2}{}^2\end{array}\right)&
R_{\mathfrak{V}}^{\mathfrak{H}}=\left(\begin{array}{cc}R_{XY3}{}^1&R_{XY4}{}^1\\R_{XY3}{}^2&R_{XY4}{}^2\end{array}\right)\\[.13in]
R_{\mathfrak{H}}^{\mathfrak{V}}=\left(\begin{array}{cc}R_{XY1}{}^3&R_{XY2}{}^3\\R_{XY1}{}^4&R_{XY2}{}^4\end{array}\right)&
R_{\mathfrak{V}}^{\mathfrak{V}}=\left(\begin{array}{cc}R_{XY3}{}^3&R_{XY4}{}^3\\R_{XY3}{}^4&R_{XY4}{}^4\end{array}\right)
\end{array}\right)\,.$$

The following result follows by a direct computation.

\begin{lemma}\label{L6.1}
Let $\mathcal{N}:=(T^*M,g_{T,\affcon,\Phi})$ where $T=\partial_{x^1}\otimes dx^2$.
\begin{enumerate}
\item $R_{\mathfrak{V}}^{\mathfrak{H}}(X,Y)=0$ for all $(X,Y)$, i.e. $R_{abi}{}^j=0$ for $3\le i\le4$, $1\le j\le 2$.

\medbreak\item $\{R_{\mathfrak{H}}^{\mathfrak{H}}+(R_{\mathfrak{V}}^{\mathfrak{V}})^t\}(X,Y)=0$ for all $(X,Y)$ i.e.
$R_{ab1}{}^1+R_{ab3}{}^{3}=0$,\newline $R_{ab2}{}^2+R_{ab4}{}^{4}=0$, $R_{ab1}{}^2+R_{ab4}{}^{3}=0$, and $R_{ab2}{}^1+R_{ab3}{}^{4}=0$.\smallbreak\noindent
$R_{\mathfrak{H}}^\mathfrak{H}(\partial_{x^i},\partial_{x^j})=0$ for $i<j$
and $(i,j)\not\in\{(1,2),(2,3)\}$.\medbreak\noindent
$\left(\begin{array}{cc}R_{231}{}^1&R_{232}{}^1\\R_{231}{}^2&R_{232}{}^2\end{array}\right)
=\left(\begin{array}{cc}0&1\\0&0\end{array}\right)$.
\medbreak\noindent
$\left(\begin{array}{cc}R_{121}{}^1&R_{122}{}^1\\[0.05in]
R_{121}{}^2&R_{122}{}^2\end{array}\right)
=\left(\begin{array}{cc}R^\affcon_{121}{}^1
&R^\affcon_{122}{}^1\\[0.05in]
R^\affcon_{121}{}^2&R^\affcon_{122}{}^2\end{array}\right)
 -x_{3}\left(\begin{array}{cc}-\GammaD_{11}{}^2
&\GammaD_{11}{}^1-\GammaD_{12}{}^2\\[0.05in]
0&\GammaD_{11}{}^2\end{array}\right)$.
\medbreak\noindent
$\operatorname{Tr}\{R_{\mathfrak{H}}^{\mathfrak{H}}(X,Y)\}=-2(\pi^*\rho_a^\affcon)(X,Y)$.
\smallbreak\item  $\rho_{ij}=0$ if $i\ge3$ or $j\ge3$.
\medbreak\noindent\hglue -.2in
$\left(\begin{array}{cc}\rho_{11}&\rho_{21}\\\rho_{12}&\rho_{22}\end{array}\right)
=2\rho_s^\nabla+\left(\begin{array}{cc}0&2x_3\Gamma_{11}{}^2\\
2x_3\Gamma_{11}{}^2&-4x_3\Gamma_{11}{}^1-2x_4\Gamma_{11}{}^2+2x_3\Gamma_{12}{}^2+\Phi_{11}\end{array}\right)$.
\smallbreak\item ${}^g\nabla R(i,j,1,1;k)+{}^g\nabla R(i,j,2,2;k)=0$ unless $\{i,j,k\}\in\{1,2\}$.
\end{enumerate}\end{lemma}

The manifold $\mathcal{N}:=(T^*M,g_{T,\nabla,\Phi})$ is a Walker manifold;  $\mathfrak{V}:=\operatorname{Span}\{\partial_{x_3},\partial_{x_4}\}$
is a null parallel distribution of rank 2 by Equation~(\ref{E1.a}) and Equation~(\ref{E5.a}). Generically, this is the only such
distribution and $\mathfrak{V}$ is invariantly defined. We use $\mathfrak{V}$ as an additional piece of structure and redefine
$\mathfrak{H}:=TN/\mathfrak{V}$ and let $\pi:TN\rightarrow\mathfrak{H}$ be the natural projection. By Lemma~\ref{L6.1}, $\pi R(X,Y)v=0$ for
$v\in\mathfrak{V}$ and thus $\pi R(X,Y)$ descends to a well defined map that, via an abuse of notation, we continue to denote by 
$R_{\mathfrak{H}}^{\mathfrak{H}}(X,Y)$ of $\mathfrak{H}$. Let $\{X_3,X_4\}$ be a local frame for $\mathfrak{V}$. Choose $\{X_1,X_2\}$ so that
\begin{equation}\label{E6.a}
g(X_1,X_3)=g(X_2,X_4)=1\text{ and }g(X_1,X_4)=g(X_2,X_3)=0\,.
\end{equation}
We note that $\{X_1,X_2\}$ is not uniquely defined by these relations as we can add an
arbitrary element of $\mathfrak{V}$ to either $X_1$ or $X_2$ and preserve Equation~(\ref{E6.a}). However $\{\pi X_1,\pi X_2\}$ is uniquely defined
Equation~(\ref{E6.a}). And, in particular, if we take $X_3=\partial_{x_ 3}$ and $X_4=\partial_{x_ 4}$, 
then we may take $X_1=\partial_{x^1}$ and $X_2=\partial_{x^2}$.

We use Lemma~\ref{L6.1} to introduce some additional quantities.
\begin{enumerate}
\item Since $\rho(X,Y)=0$ if either $X$ or $Y$ belongs to
$\mathfrak{V}$, $\rho$ descends to a map from $\mathfrak{H}\oplus\mathfrak{H}$ to $\mathbb{R}$
that we shall denote by $\rho^{\mathfrak{H}}\in S^2(\mathfrak{H}^*)$. Let $\pi:T^*M\rightarrow M$. Since $\pi_*(\mathfrak{V})=0$,
$\pi_*:\mathfrak{H}\rightarrow TM$.
If $\GammaD_{11}{}^2=0$, if $2\GammaD_{11}{}^2=\GammaD_{12}{}^2$, and if $\Phi_{11}=0$, then $\rho^{\mathfrak{H}}=2\pi^*\rho_s^\affcon$.
\item Let $\Omega(X,Y)=\operatorname{Tr}\{R_{\mathfrak{H}}^{\mathfrak{H}}(X,Y)\}$. Then $\Omega(X,Y)=0$ if either $X$ or $Y$ belongs
to $\mathfrak{V}$ so $\Omega$ descends to an alternating bilinear map from $\mathfrak{H}\oplus\mathfrak{H}$ to $\mathbb{R}$ that
we shall denote by $\Omega^{\mathfrak{H}}\in\Lambda^2(\mathfrak{H}^*)$. We have $\Omega^{\mathfrak{H}}=-2\pi^*\rho_a^\affcon$.
\item As $\mathfrak{V}$ is parallel,
${}^g\nabla R(X,Y;Z)$ maps $\mathfrak{V}$ to $\mathfrak{V}$. Consequently, ${}^g\nabla R(X,Y;Z)$ extends to an endomorphism
$({}^g\nabla R)^{\mathfrak{H}}(X,Y;Z)$ of $\mathfrak{H}$. 
A direct computation shows that $\operatorname{Tr}\{({}^g\nabla R)^{\mathfrak{H}}(X,Y;Z)\}=0$
if $X$, $Y$, or $Z$ belongs to $\mathfrak{V}$. We may therefore regard
$\operatorname{Tr}\{({}^g\nabla R)^{\mathfrak{H}}(X,Y;Z)\}\in\Lambda^2(\mathfrak{H})\otimes\mathfrak{H}^*$.
Assuming that $\Omega^{\mathfrak{H}}\ne0$, we may decompose
$\operatorname{Tr}\{({}^g\nabla R)^{\mathfrak{H}}\}=\omega^{\mathfrak{H}}\otimes\Omega^{\mathfrak{H}}$
for $\omega^{\mathfrak{H}}\in\mathfrak{H}^*$.  Moreover, one has $d\omega^\mathfrak{H}=\Omega^\mathfrak{H}$.
\end{enumerate}

\begin{definition}\rm
Suppose that we are at a point of $\mathcal{N}$ where $\rho^{\mathfrak{H}}$ defines a non-degenerate symmetric bilinear form on $\mathfrak{H}$.
We may then define
$$
\beta_1:=\|\Omega^{\mathfrak{H}}\|^2_{\rho^{\mathfrak{H}}}
=\frac{(R_{121}{}^1+R_{122}{}^2)^2}{\rho_{11}\rho_{22}-\rho_{12}\rho_{12}}
\,.
$$
If we also assume that $\Omega^{\mathfrak{H}}\ne0$ (i.e., $\rho^\affcon_a\neq 0$)
or, equivalently, that $\beta_1\ne0$, then $\omega^{\mathfrak{H}}$ is well defined
and we may set
$$
\beta_2:=\|\omega^{\mathfrak{H}}\|^2_{\rho^{\mathfrak{H}}}\,.
$$
We have
\begin{eqnarray*}
&&\omega^{\mathfrak{H}}_1=\frac{R_{121}{}^1{}_{;1}+R_{122}{}^2{}_{;1}}{R_{121}{}^1+R_{122}{}^2},\quad
\omega^{\mathfrak{H}}_2=\frac{R_{121}{}^1{}_{;2}+R_{122}{}^2{}_{;2}}{R_{121}{}^1+R_{122}{}^2},\\
&&\beta_2:=\frac{\rho_{22}^{\mathfrak{H}}\omega^{\mathfrak{H}}_{1}\omega^{\mathfrak{H}}_{1}
+\rho_{11}^{\mathfrak{H}}\omega^{\mathfrak{H}}_2\omega^{\mathfrak{H}}_2
-2\rho_{12}^{\mathfrak{H}}\omega^{\mathfrak{H}}_{1}\omega^{\mathfrak{H}}_2}
{\rho_{11}^{\mathfrak{H}}\rho_{22}^{\mathfrak{H}}-\rho_{12}^{\mathfrak{H}}\rho_{12}^{\mathfrak{H}}}\,.
\end{eqnarray*}
\end{definition}

It is obvious from the discussion given above that $\beta_1$ and $\beta_2$ are isometry invariants of
$\mathcal{N}$ where defined. Generically, $\beta_1$ and $\beta_2$ are very complicated
expressions which involve non-trivial dependence on the fiber variables
and which involve the endomorphism $\Phi$. 

 \begin{example}\rm
  Let $\mathcal{M}$ be a Type $\mathcal{A}$-surface. Since the Ricci tensor is symmetric,
  $\beta_1=0$ whenever defined and $\beta_2$ is not defined.
 \end{example}
 
 \begin{example}\rm
 We adopt the notation of Example~\ref{E4.4} and let
 $\mathcal{M}$ be a Type $\mathcal{B}$ surface so that 
  $\mathcal{N}:=(T^*M,g_{T,\affcon,\Phi})$ is Bach flat. Note that the Ricci tensor of a Type~$\mathcal{B}$
  surface is symmetric if and only if $C_{12}{}^1+C_{22}{}^2=0$. In this setting, $\beta_1=0$.
 Take coordinates $(x^1,x^2)$ on $M$ as in Theorem~\ref{Th4.1}. 
 We firstly consider the case of a nilpotent $T$ given by
$T=\partial_{x^1}\otimes dx^2$. Then we have the following two cases
\begin{enumerate}
\item Suppose $C_{11}{}^2=0$, $C_{11}{}^1=1$, and
$\rho^\mathfrak{H}$ is non-degenerate. Then we have that
$\beta_1= (C_{12}{}^1+C_{22}{}^2)^2\Delta^{-1}$
  where
 \begin{eqnarray*}
  \Delta&=&
  2 (2-C_{12}{}^2) C_{12}{}^2 (x^1)^2 \Phi_{11}
  -4 (2-C_{12}{}^2)^2 C_{12}{}^2 x^1 x_3\\
 && -(4 C_{12}{}^2+1) (C_{12}{}^1)^2  +4 (C_{12}{}^2-2)  C_{22}{}^1 (C_{12}{}^2)^2 \\
 && -(C_{22}{}^2)^2
  +2 (1-2 (C_{12}{}^2-1) C_{12}{}^2)  C_{12}{}^1 C_{22}{}^2\,.
\end{eqnarray*}
  It now follows that $\beta_1=0$ if and only if the Ricci tensor $\rho^\affcon$ of $\mathcal{M}$ is symmetric. Moreover  $\beta_1$ is a non-zero constant if and only if  either $C_{12}{}^2=0$, in which case $\beta_1=-\frac{(C_{12}{}^1+C_{22}{}^2)^2}{(C_{12}{}^1-C_{22}{}^2)^2}$, or $C_{12}{}^2=2$, and then $\beta_1=-\frac{(C_{12}{}^1+C_{22}{}^2)^2}{( 3 C_{12}{}^1+C_{22}{}^2)^2}$. 
  Further, if $\beta_1$ is non-zero, then  $\beta_2$ is generically non-constant since
\begin{eqnarray*}
\qquad\beta_2& = &
  \{
  (C_{12}{}^2+3)^2 (x^1)^2 \Phi_{11}
  +2 (C_{12}{}^2-2) (C_{12}{}^2+3)^2 x^1 x_3\\
 && -2  (C_{12}{}^2 + 3)^2  C_{12}{}^2 C_{22}{}^1
  -2 ((C_{12}{}^2 - 1) C_{12}{}^2 + 3) (C_{22}{}^2)^2\\
&&  -2 ( (4 C_{12}{}^2 + 9) C_{12}{}^2 + 6) (C_{12}{}^1)^2\\
 && -2 ( (3 C_{12}{}^2 - 4)C_{12}{}^2 - 9) C_{12}{}^1  C_{22}{}^2 
  \}\Delta^{-1}\,,
  \end{eqnarray*}
\item Suppose $C_{11}{}^2=0$, $C_{12}{}^2=C_{11}{}^1$, and
$\rho^\mathfrak{H}$ is non-degenerate. Then we have
$ \beta_1= (C_{12}{}^1+C_{22}{}^2)^2\Delta^{-1}$  where
\begin{eqnarray*}
 \Delta&=&
  2 C_{11}{}^1 (x^1)^2 \Phi_{11}
  -4 (C_{11}{}^1)^2 x^1x_3-(C_{22}{}^2)^2\\
&&  -(4 (C_{11}{}^1)^2+1) (C_{12}{}^1)^2
 -4C_{11}{}^1 C_{22}{}^1+2 C_{12}{}^1 C_{22}{}^2 \,.
\end{eqnarray*}
  Therefore $\beta_1=0$ if and only if the Ricci tensor of $\mathcal{M}$ is symmetric. Moreover, one has that  $\beta_1$ is a non-zero constant  if and only if   $C_{11}{}^1=0$, in which case $\beta_1=-\frac{(C_{12}{}^1+C_{22}{}^2)^2}{(C_{12}{}^1-C_{22}{}^2)^2}$. Furthermore, if  
  $\beta_1\neq 0$, then  
 \begin{eqnarray*}
 \beta_2 &= &
  \{
  4 (C_{11}{}^1 + 1)^2 (x^1)^2 \Phi_{11}
  -8  (C_{11}{}^1 + 1)^2   C_{11}{}^1  x^1x_3\\
&& -2 (C_{11}{}^1 + 2) (C_{22}{}^2)^2
  -8 (C_{11}{}^1 + 1)^2 C_{22}{}^1\\
&&-2 (C_{11}{}^1 (8 C_{11}{}^1 + 9) + 2) (C_{12}{}^1)^2
  +4 (3 C_{11}{}^1 + 2) C_{12}{}^1 C_{22}{}^2
  \}\Delta^{-1}\,.
\end{eqnarray*}
\end{enumerate}  
  \medskip
  Let $T=\partial_{x^2}\otimes dx^1$. 
  Proceeding in a completely analogous way as in Lemma~\ref{L6.1}, one constructs the invariants
  $\beta_1$ and $\beta_2$.
  Example~\ref{E4.4} now leads the following two possibilities.
 \begin{enumerate}
 \item Suppose $C_{22}{}^1=0$, $C_{22}{}^2=0$, and $\rho^\mathfrak{H}$ is non-degenerate.
 One then has that
  $ \beta_1= (C_{12}{}^1)^2\Delta^{-1}$
  where
\begin{eqnarray*}
  \Delta&=&
  (C_{12}{}^1)^2 
  \{
  -2 (x^1)^2 \Phi_{22}
  -4 C_{12}{}^1 x^1x_4\\
 && -4 C_{11}{}^1 C_{12}{}^2
  +4 C_{11}{}^2 C_{12}{}^1-1
  \}\,.
\end{eqnarray*}
  One now checks that  $\beta_1$ is never constant in this case. Moreover, if $\rho^\mathfrak{H}$ is non-degenerate and $\beta_1\neq 0$, then
\begin{eqnarray*}
  \beta_2& = &
  (C_{12}{}^1)^2
  \{
  (x^1)^2 \Phi_{22}
  +2 C_{12}{}^1 x^1x_4
  -12 C_{12}{}^2
  -2 C_{11}{}^2 C_{12}{}^1\\
 &&\qquad -4
  - 2 (C_{11}{}^1)^2
  -8 (C_{12}{}^2)^2
  -6 (C_{12}{}^2+1) C_{11}{}^1
  \}\Delta^{-1}\,.
\end{eqnarray*}
  It follows that $\beta_2$ is constant if and only if $2 C_{11}{}^1 + 4 C_{12}{}^2+3=0$, in which case $\beta_2=-\frac{1}{2}$.
\item Suppose $C_{22}{}^1=0$ and $C_{22}{}^2=C_{12}{}^1$. Then $\rho^\mathfrak{H}$ is non-degenerate if and only if $C_{12}{}^1 C_{12}{}^2\neq 0$. In this case one has
\begin{eqnarray*} 
\qquad\beta_1&=& -(C_{12}{}^2)^{-2}\,,
  \\
  \beta_2&=&- \left((x^1)^2 \Phi_{22}-2 C_{12}{}^1 x^1x_4
  -4 (C_{12}{}^2)^2-2 C_{12}{}^2 \right)(C_{12}{}^2)^{-2}\,.
\end{eqnarray*}
  In contrast with the previous cases, $\beta_1$ is constant while $\beta_2$ is never constant.
\end{enumerate} \end{example}

\begin{remark}\rm The fact that 
$\operatorname{Tr}(R_\mathfrak{H}^\mathfrak{H})\in\Lambda^2(\mathfrak{H}^*)$,
$\omega^\mathfrak{H}=
\operatorname{Tr}({}^g\nabla R)^\mathfrak{H}/\Omega^\mathfrak{H}\in\mathfrak{H}^*$, and
$d\omega^\mathfrak{H}=\Omega^\mathfrak{H}$ is, of course, not true for a general 
Walker manifold. This observation perhaps can be
useful in studying when a general Walker manifold is one of our special
examples. All of these are pull-backs of similar identities on the base.
\end{remark}

\end{document}